\documentclass{amsart}
\usepackage[demo]{graphicx}
\usepackage{caption}
\usepackage{subcaption}
\usepackage{graphicx, amsmath, amsfonts, epsfig, latexsym, amssymb, amsthm}
\usepackage{tikz}
\usepackage{tkz-berge}
\usepackage{float}

\usetikzlibrary {positioning}
\definecolor {processblue}{cmyk}{0.96,0,0,0}
\newtheorem{thm}{Theorem}[section]
\theoremstyle{definition}
\newtheorem{cor}[thm]{Corollary}
\newtheorem{prop}[thm]{Proposition}
\newtheorem{defn}[thm]{Definition}
\newtheorem{lem}[thm]{Lemma}

\newtheorem{Note}[thm]{Notation}

\newtheorem{ex}[thm]{Example}

\numberwithin{equation}{section}
\begin{document}
\title[The $i$-extended ideal-based cozero-divisor graph of a commutative ring]
{The $i$-extended ideal-based cozero-divisor graph of a commutative ring}
\author{ Faranak Farshadifar}
\address{Department of Mathematics Education, Farhangian University, P.O. Box 14665-889, Tehran, Iran.}
\email{f.farshadifar@cfu.ac.ir}


\begin{abstract}
Let $R$ be a commutative ring with identity and let $J$ be an ideal of $R$. In this paper,  for a positive integer $i$, we introduce and investigate the notion of the $i$-extended ideal-based cozero-divisor graph of $R$. This graph, denoted by $\overline{\Gamma''}_{Ji}(R)$,
is a simple graph of $R$ whose vertex set is $\{x \in R\setminus J\: |\: xR+J\not=R \}$. Two distinct vertices $x$ and $y$ are adjacent if and only if $x^m \not \in y^nR+J$ and $y^n \not \in x^mR+J$ for some positive integers $m$ and $n$ with $n\leq i$ and $m\leq i$.
\end{abstract}
\maketitle
\section{Introduction}
\noindent
Throughout this paper, $R$ will denote a commutative ring with identity and $\Bbb Z$ will denote the ring of of integers. Also, the Jacobson radical of $R$ will denote by $Jac(R) $.

A \emph{graph} $G$ is defined as the pair $(V(G),E(G))$, where $V(G)$ is the set of vertices of $G$ and $E(G)$ is the set of edges of $G$. For two distinct vertices $a$ and $b$ of $V(G)$, the notation $a-b$ means that $a$ and $b$ are adjacent. A graph $G$ is said to be \emph{complete} if $a-b$ for all distinct $a, b\in V(G)$. A graph
$G$ is said to be an \emph{empty graph} if $E(G) =\emptyset$. Note by this definition that a graph may be empty even if $V (G)\not =\emptyset$.

In \cite{1}, the authors introduced and investigated the \emph{cozero-divisor graph} $\Gamma'(R)$ of $R$, in which the vertices are precisely the non-zero, non-unit elements of $R$ and two distinct vertices $x$ and $y$ are adjacent if and only if $x \not \in yR$ and $y \not \in xR$. Let $J$ be an ideal of $R$. The authors in \cite{00}, introduced and studied a generalization of cozero-divisor graph $\acute{\Gamma}_J(R)$ of $R$ with vertices $\{x \in R  \setminus Ann_R(J)\: |\: xJ \neq J \}$ and two distinct vertices $x$ and $y$ are adjacent if and only if $x \not \in yJ$ and $y \not \in xJ$. In fact, $\acute{\Gamma}_J(R)$ is a generalization of cozero-divisor graph introduced in \cite{1} when $J= R$. Farshadifar, in \cite{403},  introduced and investigated a new \textit{generalization of cozero-divisor graph with respect to $J$}, denoted by $\Gamma''_J(R)$, which is
a simple graph with vertices $\{x \in R\setminus J\: |\: xR+J\not=R \}$ and two distinct vertices $x$ and $y$ are adjacent if and only if  $x \not \in yR+J$ and $y \not \in xR+J$.  In \cite{14051}, the present author introduced and studied the \textit{extended ideal based
cozero-divisor graph of} $R$, denoted by $\overline{\Gamma''}_J(R)$, which is
a simple graph $\overline{\Gamma''}_J(R)$ of $R$ with vertices $\{x \in R\setminus J\: |\: xR+J\not=R \}$. The distinct vertices $x$ and $y$ are adjacent if and only if $x^m \not \in y^nR+J$ and $y^n \not \in x^mR+J$ for some positive integers $n$ and $m$.

In \cite{Bennis2021}, the authors introduced and studied
a parameterized family of graphs $\{\overline{\Gamma}_i(R)\}_{i\in \Bbb N^*}$, for $R$, which
reveals more of the relationship between powers of zero-divisors as follows:
For a positive integer $i$, the \textit{$i$-extended zero-divisor graph} of $R$, is
a simple graph denoted by $\overline{\Gamma}_i(R)$ with vertex set $Z^*(R)$ and such that two distinct
vertices $x$ and $y$ are joined by an edge if there exist two positive integers $n\leq i$ and $m\leq i$ such that $x^ny^m = 0$ with $x^n \not= 0$ and $y^m \not= 0$ \cite{Bennis2021}.

Let $J$ be an ideal of $R$.
In this paper, we introduce and investigate
the notion of the \textit{$i$-extended ideal-based cozero-divisor graph of} $R$ which is
a simple graph $\overline{\Gamma''}_{Ji}(R)$ of $R$ with vertices $\{x \in R\setminus J\: |\: xR+J\not=R \}$. The distinct vertices $x$ and $y$ are adjacent if and only if $x^m \not \in y^nR+J$ and $y^n \not \in x^mR+J$ for some positive integers $m$ and $n$ with $n\leq i$ and $m\leq i$. This can be regarded as a dual of the $i$-extended zero-divisor graph of $R$ when $I=0$.

\section{Main Results}
\begin{defn}\label{2.1}
Let $J$ be an ideal of $R$. For a positive integer $i$, the \textit{$i$-extended ideal based
cozero-divisor graph of $R$}, is
a simple graph $\overline{\Gamma''}_{Ji}(R)$ of $R$ with vertices $\{x \in R\setminus J\: |\: xR+J\not=R \}$. The distinct vertices $x$ and $y$ are adjacent if and only if $x^m \not \in y^nR+J$ and $y^n \not \in x^mR+J$ for some positive integers $m$ and $n$ with $n\leq i$ and $m\leq i$. This can be regarded as a dual of the $i$-extended zero-divisor graph of $R$ when $I=0$.
\end{defn}

\begin{Note}
Let $J$ be an ideal of $R$. Clearly, if $J$ is a maximal ideal of $R$ or $J=R$, then $\overline{\Gamma''}_{Ji}(R) = \emptyset$ for each $i \in \Bbb N$. So, in the rest of this paper $J$ is a proper non-maximal ideal of $R$. We denote the set of vertices of
 $\overline{\Gamma''}_{Ji}(R)$ by $V$ for each $i \in \Bbb N$.
 \end{Note}

Let $J$ be an ideal of $R$.
Clearly, the 1-extended cozero-divisor graph of $R$ is equal to $\Gamma''_{J}(R)$, i.e., $\overline{\Gamma''}_{J1}(R)=\Gamma''_{J}(R)$. In addition, one can see that the family $\{\overline{\Gamma''}_{Ji}(R)\}_{i\in \Bbb N}$ forms a
filtration of the extended cozero-divisor graph $\overline{\Gamma''}_{J}(R)$. Namely, $\overline{\Gamma''}_{J}(R)=\bigcup_{i \in \Bbb N}\overline{\Gamma''}_{Ji}(R)$.

\begin{ex}\label{2.14}
Consider the ring $R = \Bbb Z_2[X, Y]$. Clearly, $X, XY \in V$. In addition,
$XY \not \in RX^2$ and $X^2 \not \in RXY$. Thus $X$ and $XY$ are adjacent in $\overline{\Gamma''}_{02}(R)$. But $XY \in RX$ implies that $X$ and $XY$ are not adjacent in $\overline{\Gamma''}_{01}(R)$.
\end{ex}

In \cite{Bennis2021}, it is shown that $\overline{\Gamma}_{i}(\Bbb Z_{p^n})$ is a complete graph for all $i \geq \frac{n}{2}$. But in the Proposition \ref{2.12}, we can see that  $\overline{\Gamma''}_{0i}(\Bbb Z_{p^n})$ is an empty graph for each $i \in \Bbb N$.
\begin{prop}\label{2.12}
Let $n \in \Bbb N$ and $p$ be a prime number. Then $\overline{\Gamma''}_{0i}(\Bbb Z_{p^n})$ is an empty graph for each $i \in \Bbb N$.
\end{prop}
\begin{proof}
Let $x, y \in V(\overline{\Gamma''}_{0i}(\Bbb Z_{p^n}))$. Then $x=p^tz$ and $y=p^sw$, where $\mathrm{gcd}(p,z)=1$ and $\mathrm{gcd}(p,w)=1$.  As $\mathrm{gcd}(p^t,z)=1$ (resp., $\mathrm{gcd}(p^s,w)=1$), we have $z\Bbb Z_{p^n}=\Bbb Z_{p^n}$ (resp., $w\Bbb Z_{p^n}=\Bbb Z_{p^n}$).
If $t=s$, then $x \in y\Bbb Z_{p^n}$ and $y \in x\Bbb Z_{p^n}$. Hence $x$ is not adjacent to $y$. So suppose that $t<s$. Then
$p^s \in p^t\Bbb Z_{p^n}=p^tz\Bbb Z_{p^n}$. Hence, $y=p^sw \in  p^tz\Bbb Z_{p^n}=x\Bbb Z_{p^n}$. So, $x$ is not adjacent to $y$. Therefore $E(\overline{\Gamma''}_{0i}(\Bbb Z_{p^n}))=\emptyset$. Thus
$\overline{\Gamma''}_{0i}(\Bbb Z_{p^n})$ is an empty graph for each $i \in \Bbb N$.
\end{proof}

\begin{prop}\label{2.9}
Let $p$ and $q$ be two primes and $n \in \Bbb N$. Then $\overline{\Gamma''}_{0(n-1)}(\Bbb Z_{p^nq})\not=\overline{\Gamma''}_{0n}(\Bbb Z_{p^nq})$ for all $n>1$.
\end{prop}
\begin{proof}
Let $n>1$.
Since $p^{n-1}q \not \in p^n\Bbb Z_{p^nq}$ and $p^n \not \in p^{n-1}q\Bbb Z_{p^nq}$, we have $p^{n-1}q$ and $p$ are adjacent in
$\overline{\Gamma''}_{0n}(\Bbb Z_{p^nq})$. But since $p^{n-1}q \in p^k\Bbb Z_{p^nq}$ for each $k<n$, we have
$p^{n-1}q$ and $p$ are not adjacent in $\overline{\Gamma''}_{0(n-1)}(\Bbb Z_{p^nq})$. Therefore, $\overline{\Gamma''}_{0(n-1)}(\Bbb Z_{p^nq})\not=\overline{\Gamma''}_{0n}(\Bbb Z_{p^nq})$.
\end{proof}

\begin{prop}
Let $J$ be an ideal of $R$. Then we have the following.
\begin{itemize}
\item [(a)] If $J$ is a prime ideal of $R$, then $\overline{\Gamma''}_{Ji}(R)$ is not a complete graph for each $i \in \Bbb N$.
\item [(b)] $\overline{\Gamma''}_{Jk}(R)$ is a subgraph of $\overline{\Gamma''}_{Jn}(R)$ for each $k,n \in \Bbb N$ with $k<n$.
\end{itemize}
\end{prop}
\begin{proof}
(a) Let $J$ be a prime ideal of $R$. Assume that $x \in V$. Then $x \not \in J$ and $Rx+J\not =R$. Thus $Rx^2+J\not =R$. Since $J$ is a prime ideal of $R$, $x^2\not \in I$. Hence, $x^2 \in V$. If $x^2=x$, then
$x(1-x)=0\in J$ implies that $x \in J$ or $Rx+J=R$. These contradictions show that  $x^2\not=x$.
Now, as $x$ and $x^2$ are not adjacent in $\overline{\Gamma''}_{Ji}(R)$ for each $i \in \Bbb N$, we have
$\overline{\Gamma''}_{Ji}(R)$ is not a complete graph for each $i \in \Bbb N$.

(b) This is clear.
\end{proof}

\begin{thm}\label{2.99}
Let $p$ and $q$ be two primes and $n \in \Bbb N$. Then $\overline{\Gamma''}_{0n}(\Bbb Z_{p^nq})$  is a complete tripartite graph
\end{thm}
\begin{proof}
Let $x \in V$. Then $x=kp^aq^b$, where $\mathrm{gcd}(k,p)=1$, $\mathrm{gcd}(k,q)=1$, $a,b,k \in \Bbb N$, and $a\leq n$.
Let $V_1=\{x\in V:a=0\}$, $V_2=\{x\in V:b=0\}$, and $V_3=\{x\in V:a\not=0\ and \ b\not=0 \}$.
Then $V=V_1 \cup V_2 \cup V_3$. One can see that $\overline{\Gamma''}_{0n}(\Bbb Z_{p^nq})$ is a complete tripartite graph with parts $V_1$, $V_2$, and $V_3$.
\end{proof}

\begin{ex}\label{2.2}
Let $R=\Bbb Z_{12}$ and $I=0$. Then as we can see in the following figures,  $\overline{\Gamma''}_{02}(R)=\overline{\Gamma''}_{0i}(R)$ for each  positive integer $i\geq 3$.
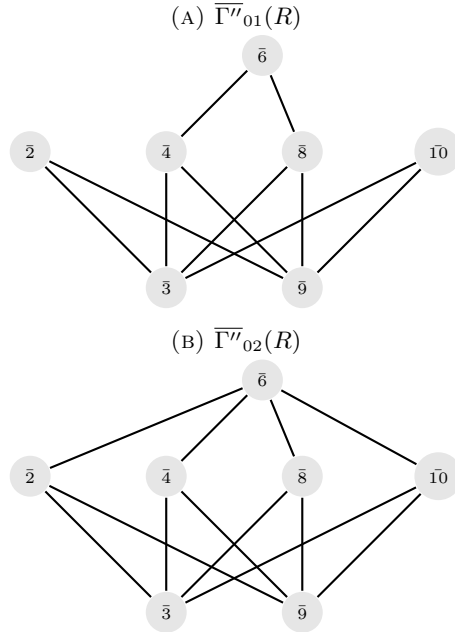
\begin{figure}[H]

\begin{subfigure}[h]{0.5\textwidth}
\centering
{$\overline{\Gamma''}_{01}(R)$}
\begin{tikzpicture}[auto,node distance=1.8cm,
  thick,main node/.style={circle,fill=black!10,font=\sffamily\tiny\bfseries}]
\node[main node] (1) {$\bar{6}$};
\node[main node] (2) [below left of=1] {$\bar{4}$};
\node[main node] (3) [left of=2] {$\bar{2}$};
\node[main node] (4) [right of=2] {$\bar{8}$};
\node[main node] (5) [right of=4] {$\bar{10}$};
\node[main node] (6) [below of=2] {$\bar{3}$};
\node[main node] (7) [right of=6] {$\bar{9}$};
\path[every node/.style={font=\sffamily\small}]
    (1) edge node [left] {} (2)
    (1) edge node [left] {} (4)
    (3) edge node [left] {} (6)
    (3) edge node [left] {} (7)
    (2) edge node [left] {} (6)
    (2) edge node [left] {} (7)
    (4) edge node [left] {} (7)
    (4) edge node [left] {} (6)
    (5) edge node [left] {} (6)
      (5) edge node [left] {} (7);
\end{tikzpicture}
\end{subfigure}
\begin{subfigure}[h]{0.5\textwidth}
\centering
{$\overline{\Gamma''}_{02}(R)$}
\begin{tikzpicture}[auto,node distance=1.8cm,
  thick,main node/.style={circle,fill=black!10,font=\sffamily\tiny\bfseries}]
\node[main node] (1) {$\bar{6}$};
\node[main node] (2) [below left of=1] {$\bar{4}$};
\node[main node] (3) [left of=2] {$\bar{2}$};
\node[main node] (4) [right of=2] {$\bar{8}$};
\node[main node] (5) [right of=4] {$\bar{10}$};
\node[main node] (6) [below of=2] {$\bar{3}$};
\node[main node] (7) [right of=6] {$\bar{9}$};
\path[every node/.style={font=\sffamily\small}]
    (1) edge node [left] {} (2)
    (1) edge node [left] {} (4)
    (3) edge node [left] {} (6)
    (3) edge node [left] {} (7)
    (2) edge node [left] {} (6)
    (2) edge node [left] {} (7)
    (4) edge node [left] {} (7)
    (4) edge node [left] {} (6)
    (5) edge node [left] {} (6)
     (1) edge node [left] {} (5)
      (1) edge node [left] {} (3)
      (5) edge node [left] {} (7);
\end{tikzpicture}
\end{subfigure}

\end{figure}
 \end{ex}

\begin{ex}\label{2.3}
Let $R=\Bbb Z_{24}$ and $I=0$. Then as we can see in the following figures,  $\overline{\Gamma''}_{03}(R)=\overline{\Gamma''}_{0i}(R)$ for each  positive integer $i\geq 4$.
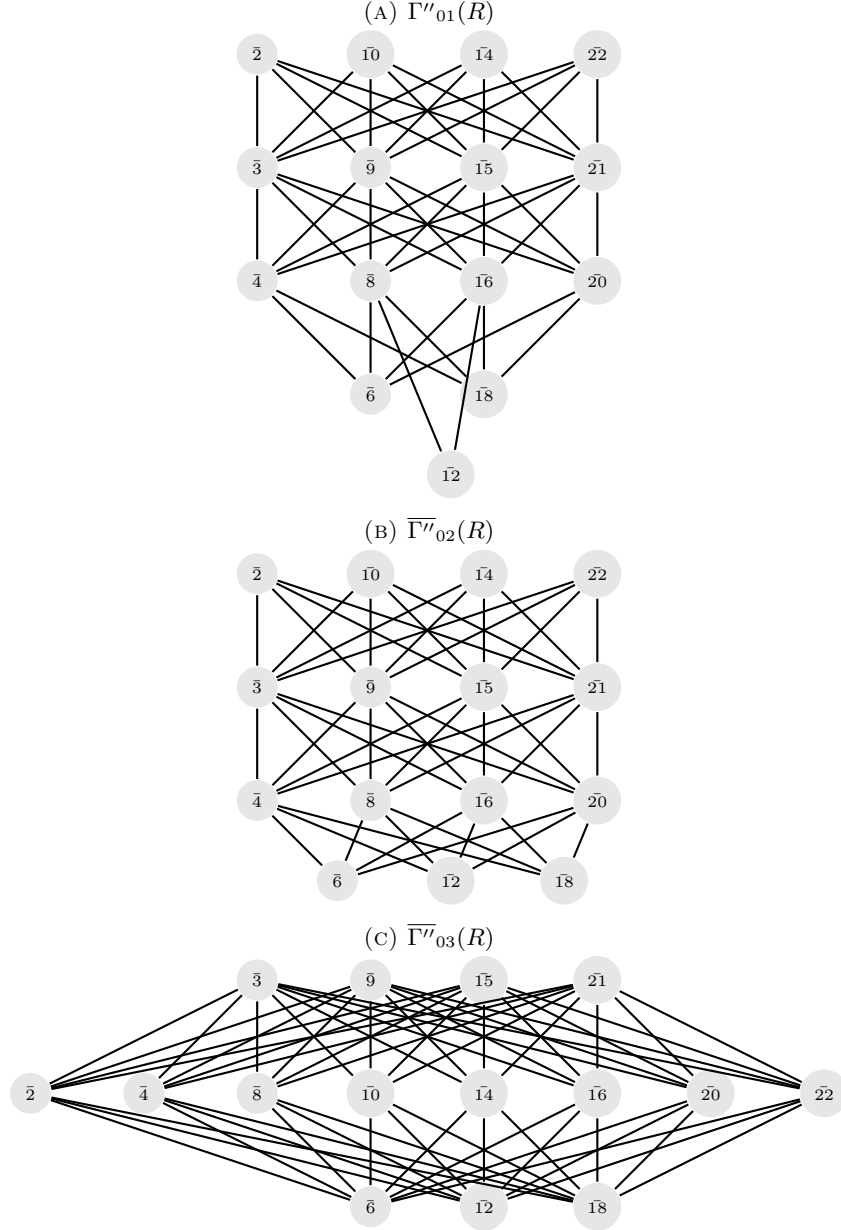
\begin{figure}[H]
\begin{subfigure}[h]{0.5\textwidth}
\centering
{$\overline{\Gamma''}_{01}(R)$}
\begin{tikzpicture}[auto,node distance=1.5cm,
  thick,main node/.style={circle,fill=black!10,font=\sffamily\tiny\bfseries}]
\node[main node] (1) {$\bar{2}$};
\node[main node] (2) [right of=1] {$\bar{10}$};
\node[main node] (3) [right of=2] {$\bar{14}$};
\node[main node] (4) [right of=3] {$\bar{22}$};
\node[main node] (5) [below of=1] {$\bar{3}$};
\node[main node] (6) [right of=5] {$\bar{9}$};
\node[main node] (7) [right of=6] {$\bar{15}$};
\node[main node] (8) [right of=7] {$\bar{21}$};
\node[main node] (9) [below of=5] {$\bar{4}$};
\node[main node] (10) [right of=9] {$\bar{8}$};
\node[main node] (11) [right of=10] {$\bar{16}$};
\node[main node] (12) [right of=11] {$\bar{20}$};
\node[main node] (13) [below of=10] {$\bar{6}$};
\node[main node] (14) [right of=13] {$\bar{18}$};
\node[main node] (15) [below right of=13] {$\bar{12}$};
\path[every node/.style={font=\sffamily\small}]
    (1) edge node [left] {} (5)
    (1) edge node [left] {} (6)
        (1) edge node [left] {} (7)
         (1) edge node [left] {} (8)
          (2) edge node [left] {} (5)
           (2) edge node [left] {} (6)
            (2) edge node [left] {} (7)
             (2) edge node [left] {} (8)
              (3) edge node [left] {} (5)
               (3) edge node [left] {} (6)
                (3) edge node [left] {} (7)
                 (3) edge node [left] {} (8)
                  (4) edge node [left] {} (5)
                   (4) edge node [left] {} (6)
                    (4) edge node [left] {} (7)
                     (4) edge node [left] {} (8)
                      (9) edge node [left] {} (14)
                       (9) edge node [left] {} (5)
                        (9) edge node [left] {} (6)
                         (9) edge node [left] {} (7)
                          (9) edge node [left] {} (8)
                         (9) edge node [left] {} (13)
                         (10) edge node [left] {} (5)
       (10) edge node [left] {} (6)
        (10) edge node [left] {} (7)
         (10) edge node [left] {} (8)
          (11) edge node [left] {} (5)
           (11) edge node [left] {} (6)
            (11) edge node [left] {} (7)
             (11) edge node [left] {} (8)
              (12) edge node [left] {} (5)
               (12) edge node [left] {} (6)
                (12) edge node [left] {} (7)
                 (12) edge node [left] {} (8)
                  (13) edge node [left] {} (10)
                   (13) edge node [left] {} (11)
                    (13) edge node [left] {} (12)
            (14) edge node [left] {} (10)
                   (14) edge node [left] {} (11)
                    (14) edge node [left] {} (12)
                                        (15) edge node [left] {} (10)
                    (15) edge node [left] {} (11);
\end{tikzpicture}
\end{subfigure}
\begin{subfigure}[h]{0.5\textwidth}
\centering
{$\overline{\Gamma''}_{02}(R)$}
\begin{tikzpicture}[auto,node distance=1.5cm,
  thick,main node/.style={circle,fill=black!10,font=\sffamily\tiny\bfseries}]
\node[main node] (1) {$\bar{2}$};
\node[main node] (2) [right of=1] {$\bar{10}$};
\node[main node] (3) [right of=2] {$\bar{14}$};
\node[main node] (4) [right of=3] {$\bar{22}$};
\node[main node] (5) [below of=1] {$\bar{3}$};
\node[main node] (6) [right of=5] {$\bar{9}$};
\node[main node] (7) [right of=6] {$\bar{15}$};
\node[main node] (8) [right of=7] {$\bar{21}$};
\node[main node] (9) [below of=5] {$\bar{4}$};
\node[main node] (10) [right of=9] {$\bar{8}$};
\node[main node] (11) [right of=10] {$\bar{16}$};
\node[main node] (12) [right of=11] {$\bar{20}$};
\node[main node] (13) [below right of=9] {$\bar{6}$};
\node[main node] (14) [right of=13] {$\bar{12}$};
\node[main node] (15) [right of=14] {$\bar{18}$};
\path[every node/.style={font=\sffamily\small}]
    (1) edge node [left] {} (5)
    (1) edge node [left] {} (6)
        (1) edge node [left] {} (7)
         (1) edge node [left] {} (8)
          (2) edge node [left] {} (5)
           (2) edge node [left] {} (6)
            (2) edge node [left] {} (7)
             (2) edge node [left] {} (8)
              (3) edge node [left] {} (5)
               (3) edge node [left] {} (6)
                (3) edge node [left] {} (7)
                 (3) edge node [left] {} (8)
                  (4) edge node [left] {} (5)
                   (4) edge node [left] {} (6)
                    (4) edge node [left] {} (7)
                     (4) edge node [left] {} (8)
                      (9) edge node [left] {} (14)
                       (9) edge node [left] {} (5)
                        (9) edge node [left] {} (6)
                         (9) edge node [left] {} (7)
                          (9) edge node [left] {} (8)
                         (9) edge node [left] {} (13)
                         (10) edge node [left] {} (5)
       (10) edge node [left] {} (6)
        (10) edge node [left] {} (7)
         (10) edge node [left] {} (8)
          (11) edge node [left] {} (5)
           (11) edge node [left] {} (6)
            (11) edge node [left] {} (7)
             (11) edge node [left] {} (8)
              (12) edge node [left] {} (5)
               (12) edge node [left] {} (6)
                (12) edge node [left] {} (7)
                 (12) edge node [left] {} (8)
                  (13) edge node [left] {} (10)
                   (13) edge node [left] {} (11)
                    (13) edge node [left] {} (12)
            (14) edge node [left] {} (10)
                   (14) edge node [left] {} (11)
                    (14) edge node [left] {} (12)
                                        (15) edge node [left] {} (10)
                    (15) edge node [left] {} (11)
                    (15) edge node [left] {} (9)
                    (15) edge node [left] {} (12);
\end{tikzpicture}
\end{subfigure}

\centering
\begin{subfigure}[b]{1\textwidth}
\centering {$\overline{\Gamma''}_{03}(R)$}

\begin{tikzpicture}[auto,node distance=1.5cm,
  thick,main node/.style={circle,fill=black!10,font=\sffamily\tiny\bfseries}]
\node[main node] (1) {$\bar{3}$};
\node[main node] (2) [right of=1] {$\bar{9}$};
\node[main node] (3) [right of=2] {$\bar{15}$};
\node[main node] (4) [right of=3] {$\bar{21}$};
\node[main node] (5) [below of =1] {$\bar{8}$};
\node[main node] (6) [left of=5] {$\bar{4}$};
\node[main node] (7) [left of=6] {$\bar{2}$};
\node[main node] (8) [right of=5] {$\bar{10}$};
\node[main node] (9) [right of=8] {$\bar{14}$};
\node[main node] (10) [right of=9] {$\bar{16}$};
\node[main node] (11) [right of=10] {$\bar{20}$};
\node[main node] (12) [right of=11] {$\bar{22}$};
\node[main node] (13) [below of=8 ] {$\bar{6}$};
\node[main node] (14) [right of=13 ] {$\bar{12}$};
\node[main node] (15) [right of=14] {$\bar{18}$};
\path[every node/.style={font=\sffamily\small}]
(1) edge node [left] {} (5)
(1) edge node [left] {} (6)
(1) edge node [left] {} (7)
(1) edge node [left] {} (8)
(1) edge node [left] {} (9)
(1) edge node [left] {} (10)
(1) edge node [left] {} (11)
    (1) edge node [left] {} (12)
        (2) edge node [left] {} (5)
(2) edge node [left] {} (6)
(2) edge node [left] {} (7)
(2) edge node [left] {} (8)
(2) edge node [left] {} (9)
(2) edge node [left] {} (10)
(2) edge node [left] {} (11)
    (2) edge node [left] {} (12)
         (3) edge node [left] {} (5)
(3) edge node [left] {} (6)
(3) edge node [left] {} (7)
(3) edge node [left] {} (8)
(3) edge node [left] {} (9)
(3) edge node [left] {} (10)
(3) edge node [left] {} (11)
    (3) edge node [left] {} (12)
         (4) edge node [left] {} (5)
(4) edge node [left] {} (6)
(4) edge node [left] {} (7)
(4) edge node [left] {} (8)
(4) edge node [left] {} (9)
(4) edge node [left] {} (10)
(4) edge node [left] {} (11)
    (4) edge node [left] {} (12)
         (13) edge node [left] {} (5)
(13) edge node [left] {} (6)
(13) edge node [left] {} (7)
(13) edge node [left] {} (8)
(13) edge node [left] {} (9)
(13) edge node [left] {} (10)
(13) edge node [left] {} (11)
    (13) edge node [left] {} (12)
     (14) edge node [left] {} (5)
(14) edge node [left] {} (6)
(14) edge node [left] {} (7)
(14) edge node [left] {} (8)
(14) edge node [left] {} (9)
(14) edge node [left] {} (10)
(14) edge node [left] {} (11)
    (14) edge node [left] {} (12)
        (15) edge node [left] {} (5)
(15) edge node [left] {} (6)
(15) edge node [left] {} (7)
(15) edge node [left] {} (8)
(15) edge node [left] {} (9)
(15) edge node [left] {} (10)
(15) edge node [left] {} (11)
    (15) edge node [left] {} (12);
\end{tikzpicture}
\end{subfigure}
\end{figure}
 \end{ex}

\begin{defn}\label{2.4}
Let $J$ be an ideal of $R$ and $x \in R$. We say that $x$ is a \textit{conilpotent relative to $J$ element} of $R$ if
$1-x \not \in Rx^n+J$ and $x^n \not \in R(1-x)+J$ for some positive integer $n$.
For a conilpotent relative to $J$ element $x$ of $R$, $\xi (x)$ denotes the smallest positive integer $k$ such that $1-x \not \in Rx^k+J$ and $x^k \not \in R(1-x)+J$. In addition, we set $\xi (R)=Sup \{\xi (x): x \ is \ a \ conilpotent  \ relative \ to \ J\ element \ of \ R \}$.
\end{defn}

\begin{thm}\label{2.8}
Let $J$ be an ideal of $R$. If $\overline{\Gamma''}_{J2}(R)=\overline{\Gamma''}_{J1}(R)$,  then $\xi (R)\not=2i$ for each $i \in \Bbb N$.
\end{thm}
\begin{proof}
Let $\overline{\Gamma''}_{J2}(R)=\overline{\Gamma''}_{J1}(R)$.
Assume contrary that $\xi (R)=2i$ for some $i \in \Bbb N$. Then there is $x \in R$ such that $\xi (x)=2i$. Thus $1-x \not \in Rx^{2i}+J$ and $x^{2i} \not \in R(1-x)+J$. Hence,  $1-x \not \in R(x^i)^2+J$ and $(x^i)^2 \not \in R(1-x)+J$. Thus $1-x$ and $x^i$ are adjacent in $\overline{\Gamma''}_{J2}(R)$.  Since $\overline{\Gamma''}_{J2}(R)=\overline{\Gamma''}_{J1}(R)$, we have $1-x$ and $x^i$ are adjacent in $\overline{\Gamma''}_{J1}(R)$. This implies that  $1-x \not \in Rx^i+J$ and $x^i \not \in R(1-x)+J$, which is a contradiction since $\xi (x)=2i$. Therefore, $\xi (R)\not=2i$.
\end{proof}

\begin{prop}\label{2.5}
Let $J$ be an ideal of $R$ such that $J\subseteq Jac(R) $ and $x$ be a non-unit of $R\setminus Jac(R) $ such that $x^n=x^{n+1}$ for some $n \in \Bbb N$. Then $1-x \not \in Rx^n+J$ and $x^n \not \in R(1-x)+J$. That is, $x$ is a conilpotent relative to $J$ element of $R$.
\end{prop}
\begin{proof}
First assume contrary that $1-x \in Rx^n+J$.  Then $1-x=rx^n+a$ for some $r \in R$ and $a \in J$. Thus
$x(rx^{n-1}+1)=1-a$. Since $a \in J\subseteq Jac(R) $, $1-a$ has an inverse $b$, say in $R$. Therefore, $bx(rx^{n-1}+1)=1$. It follows that $Rx=R$. Which is a contradiction since $x$ is a non-unit of $R$. Thus $1-x \not\in Rx^n+J$.
Now, assume contrary that $x^n \not \in R(1-x)+J$. Then $x^n =r(1-x)+a$ for some $r \in R$ and $a \in J$. Thus
$x^n =r-rx+a$ and hence $x^{2n} =rx^n-rx^{n+1}+ax^n$. It follows that $x^{2n} =rx^n-rx^n+ax^n=ax^n \in J$. Since $J\subseteq Jac(R) $, $x^{2n}\in Jac(R) $. It follows that $x\in Jac(R) $, which is a contradiction. Thus
$x^n \not \in R(1-x)+J$, i.e., $x$ is a conilpotent relative to $J$ element of $R$.
\end{proof}

\begin{prop}\label{2.6}
Let $J$ be an ideal of $R$ and $x \in R$. Then we have the following.
\begin{itemize}
\item [(a)] If for some $n \in \Bbb N$ we have $x^{n+1}=x^n \in V$, then $1-x \in V$.
\item [(b)] If $J\subseteq Jac(R) $, $x$ is a non-unit of $R$, and $1-x\in V$, then $x^n \in V$ for each $n \in \Bbb N$.
\end{itemize}
\end{prop}
\begin{proof}
(a) Let for some $n \in \Bbb N$, $x^{n+1}=x^n \in V$. Then $x^n \not \in J$ and $Rx^n+J\not=R$. Assume contrary that $1-x\in J$ or $R(1-x)+J=R$. If $1-x\in J$, then $Rx+J=R$. Thus by \cite[Lemma 2.1]{14051}, $Rx^n+J=R$, which is a contradiction. Hence $1-x\not \in J$. If
$R(1-x)+J=R$, then $1=r(1-x)+a$ for some $r \in R$ and $a \in J$. This implies that $x^n=rx^n-rx^{n+1}+ax^n$. Since $x^{n+1}=x^n$, we have $x^n=ax^n \in J$. This contradiction shows that $R(1-x)+J\not=R$. Therefore, $1-x \in V$.

(b) Let $J\subseteq Jac(R) $, $x$ a non-unit of $R$, and $1-x\in V$.  Assume contrary that $x^n \in J$ or $Rx^n+J=R$ for some $n \in \Bbb N$. If $x^n \in J$, then $1=x^n+(1-x^n)\in J+R(1-x)$. It follows that $R(1-x)+J=R$, which is a contradiction. Hence, $x^n \not\in J$. If
$Rx^n+J=R$, then $Rx+J=R$. Hence $1-rx \in J \subseteq Jac(R)$ for some $r \in R$. It follows that $x$ is an unite, which is a contradiction. Therefore, $x^n \in V$ for each $n \in \Bbb N$.
\end{proof}

\begin{cor}\label{2.7}
Let $J$ be an ideal of $R$ such that $J\subseteq Jac(R) $ and $x$ be a non-unit of $R\setminus Jac(R) $ such that $x^n=x^{n+1}$ for some $n \in \Bbb N$. Then $x^n$ and $1-x$ are adjacent in $\Gamma''_{Jac(R)}(R) $.
\end{cor}
\begin{proof}
If $1-x \in J\subseteq Jac(R) $, then $x$ is an unite. Which is a contradiction. Thus $1-x \not\in J$. If $R(1-x)+J=R$, then $1=r(1-x)+a$ for some $r \in R$ and $a \in J$. This implies that $x^n=rx^n-rx^{n+1}+ax^n$. Since $x^{n+1}=x^n$, we have $x^n=ax^n \in J\subseteq Jac(R) $. Hence, $x \in Jac(R) $, which is a contradiction. Therefore, $1-x \in V$. Thus by Proposition \ref{2.6} (b), $x^n \in V$.
Now, the result follows from Proposition \ref{2.5}.
\end{proof}

\begin{lem}\label{2.10}
Let $J$ be an ideal of $R$. If $x^ny$ and $y$ are adjacent in $\Gamma''_J(R) $ for some $n \in \Bbb N$ with $n>1$, then $x^{n-k}y$ and $y$ are adjacent in $\Gamma''_J(R) $ for each $k<n$.
\end{lem}
\begin{proof}
Suppose that $x^ny$ and $y$ are adjacent in $\Gamma''_J(R) $ for some $n \in \Bbb N$. Then $y \not \in Rx^ny+J$ and $x^ny \not \in Ry+J$. Let $k<n$. First note that if $x^{n-k}y\in J$, then $x^ny=x^{n-k}x^ky\in J$, a contradiction. If $Rx^{n-k}y+ J=R$, then $R=Rx^{n-k}y+ J\subseteq Ry+J$, which is a contradiction. Hence $x^{n-k}y\in V$.
Assume contrary that $y \in Rx^{n-k}y+J$ or $x^{n-k}y \in Ry+J$. If $y \in Rx^{n-k}y+J$, then $x^ny \in Rx^{n-k}x^ny+Jx^n\subseteq Ry+J$, which is a contradiction. If $x^{n-k}y \in Ry+J$, then $x^ny=x^kx^{n-k}y \in Rx^ky+x^kJ\subseteq Ry+J$, which is a contradiction. Therefore, $x^{n-k}y$ and $y$ are adjacent in $\Gamma''_J(R) $ for each $k<n$.
\end{proof}

\begin{thm}\label{2.11}
Let $J$ be an ideal of $R$. If $xy$ and $y$ are adjacent in $\overline{\Gamma''}_{Ji}(R)$ for some $i \in \Bbb N$ and $y^2=y$, then $xy$ and $y$ are adjacent in $\overline{\Gamma''}_{J1}(R)$.
\end{thm}
\begin{proof}
Suppose that $xy$ and $y$ are adjacent in $\overline{\Gamma''}_{Ji}(R)$ for some $i \in \Bbb N$ and $y^2=y$. Then
there exist $n, m \in \Bbb N$ such that $n\leq i$ and $m\leq i$ with
$y^m \not \in R(xy)^n+J$ and $(xy)^n \not \in Ry^m+J$. This implies that $y \not \in Rx^ny+J$ and $x^ny \not \in Ry+J$ since $y^2=y$. Hence, $x^ny$ and $y$ are adjacent in $\overline{\Gamma''}_{J1}(R)$. Now, by Lemma \ref{2.10}, we have
$xy$ and $y$ are adjacent in $\overline{\Gamma''}_{J1}(R)$.
\end{proof}

\begin{thm}\label{2.102}
Let $J$ be an ideal of $R$ such that $J\subseteq Jac(R)$, $i \in \Bbb N$, and $Max(R) = \{\mathfrak{m_1}, \mathfrak{m_2}\}$. Then  $\overline{\Gamma''}_{Ji}(R)\setminus Jac(R)$ is a complete bipartite graph with parts $\mathfrak{m_i}\setminus Jac(R)$ for $i = 1, 2$ if and only if for every $x, y \in \mathfrak{m_i}\setminus Jac(R)$, for some $i = 1, 2$,  the ideals $Rx^n+J$ and $Ry^m+J$ are totally ordered (i.e. either
$Rx^n+J\subseteq Ry^m+J$ or $Ry^m+J\subseteq Rx^n+J$) for each $n, m \in \Bbb N$ with $m,n\leq i$.
\end{thm}
\begin{proof}
First suppose that $\overline{\Gamma''}_{Ji}(R)\setminus Jac(R)$ is a complete bipartite graph with parts $\mathfrak{m_i}\setminus Jac(R)$ for $i = 1, 2$.
Assume contrary that there exist $x, y \in \mathfrak{m_1}\setminus Jac(R)$ such that $Rx^n+I\not\subseteq Ry^m+J$ and $Ry^m+I\not\subseteq Rx^n+J$ for some $n, m \in \Bbb N$ with $m,n\leq i$. Then $x$ is adjacent to $y$ in $\mathfrak{m_1}\setminus Jac(R)$, which is a contradiction. Conversely, by assumption, one can see that for every elements $x$ and $y$ in $\mathfrak{m_i}\setminus Jac(R)$ for some $ i = 1, 2$, $x$ is not adjacent to $y$. Now, let $x \in \mathfrak{m_1}\setminus  \mathfrak{m_2}$ and  $y \in \mathfrak{m_2}\setminus  \mathfrak{m_1}$. If $x \in Ry+J$, then $x \in \mathfrak{m_2}+Jac(R)=\mathfrak{m_2}$ since $J\subseteq Jac(R)$. This contradiction shows that  $x\not \in Ry+J$. Similarly, $y\not \in Rx+J$ and so $x$ is adjacent to $y$. Thus $\overline{\Gamma''}_{Ji}(R)\setminus Jac(R)$ is a complete bipartite graph with parts $\mathfrak{m_1}\setminus  \mathfrak{m_2}$ and $\mathfrak{m_2}\setminus  \mathfrak{m_1}$.
\end{proof}

\begin{defn}\label{2.13}
Let $J$ be an ideal of $R$. For a positive integer $i$, the \textit{$i$-extended ideal based
zero-divisor graph of $R$}, is
a simple graph $\overline{\Gamma}_{Ji}(R)$ of $R$ with vertices $\{x \in R\setminus J\: |\: xy=0\ for \ some\ y \in R \setminus J \}$. The distinct vertices $x$ and $y$ are adjacent if and only if there exist two positive integers $n\leq i$ and $m\leq i$ such that $x^ny^m \in J$ with $x^n \not\in J$ and $y^m \not\in J$
\end{defn}

\begin{thm}\label{2.14}
Let $J$ be an ideal of $R$ and $\overline{\Gamma''}_{Ji}(R)$ be a complete graph for some $i \in \Bbb N$. Then $\overline{\Gamma}_{Ji}(R)$ is also a complete graph.
\end{thm}
\begin{proof}
Let $x$ and $y$ be two distinct elements in $\overline{\Gamma}_{Ji}(R)$. Assume contrary that $x$ is not adjacent to $y$ in  $\overline{\Gamma}_{Ji}(R)$. Then $x^ny^m \not \in J$ for each $n,m\leq i$. In particular, $x^iy\not \in J$.
Clearly $V(\Gamma_I(R))\subseteq V$. Thus $y \in V$ and so $Ry+J\not=R$. It follows that $Rx^iy+J\not=R$.
Therefore, $x^iy \in V$. Then $x^iy-x$ in $\overline{\Gamma''}_{Ji}(R)$ since $\overline{\Gamma''}_{Ji}(R)$ is a complete graph.
This implies that $(x^iy)^n \not \in Rx^m+J$ and $x^m \not \in R(x^iy)^n+J$ for some  $n,m\leq i$. Now,
$(x^iy)^n \not \in Rx^m+J$ implies that $in < m$. This is a contradiction since $m\leq i$. Thus $x$ is adjacent to $y$ in  $\overline{\Gamma}_{Ji}(R)$, as needed.
\end{proof}

Recall that an ideal $J$ of $R$ is said to be a \textit{semiprime ideal} if $J$ is proper ideal of $R$ and $x^n \in J$ implies that $x \in J$ for each $n \in \Bbb N$.
\begin{thm}\label{2.15}
Let $J$ be a semiprime ideal of $R$ and $i \in \Bbb N$. If $V(\overline{\Gamma}_{Ji}(R))=V$, then $\overline{\Gamma}_{Ji}(R)$ is not a complete graph. In particular, $\overline{\Gamma''}_{Ji}(R)$ is not a complete graph.
\end{thm}
\begin{proof}
Let $V(\overline{\Gamma}_{Ji}(R))=V$.
Assume contrary that $\overline{\Gamma}_{Ji}(R)$ is a complete graph.
Let $x, y \in V(\overline{\Gamma}_{Ji}(R))$. Then $x^ny^m \in J$ for some $n,m\leq i$.
Now as $J$ be a semiprime ideal of $R$, we get that $x^n(x+y^m) \not \in J$. Since $x \in V(\overline{\Gamma}_{Ji}(R))=V$, we have $Rx+J \not=R$. It follows that $Rx^n(x+y^m)+J \not=R$. Therefore, $x^n(x+y^m)\in V=V(\overline{\Gamma}_{Ji}(R))$. Thus as $\overline{\Gamma}_{Ji}(R)$ is a complete graph, we have
$$
x^{n+2}+x^{n+1}y^m=x(x^n(x+y^m))\in J.
$$
It follows that $x^{n+2}\in J$. Now since $J$ is a semiprime ideal of $R$, we have $x \in J$, which is a required contradiction. Now the last assertion follows from Theorem \ref{2.14}.
\end{proof}

\bibliographystyle{amsplain}

\end{document}